\documentclass[11pt]{article}
\usepackage[latin1]{inputenc}
\usepackage{amsfonts}
\usepackage{mathrsfs,amssymb,amsmath}
\usepackage{enumerate,textcomp}
\usepackage{amsthm}
\usepackage{pifont}
\usepackage[numbers,sort&compress]{natbib}
\usepackage{setspace}

\newtheorem{theorem}{Theorem}

\newtheorem{definition}[theorem]{Definition}

\newtheorem{lemma}[theorem]{Lemma}

\newcommand{\ve}{\varepsilon}

\usepackage{color}

\def \n  {\nonumber}

\def \mc{\mathcal}
\def \ms{\mathscr}
\def \mf{\mathfrak}

\def \ra{\rightarrow}

\def \wt{\widetilde}
\def \mb{\mathbb}
\def \ve{\varepsilon}
\def \ov{\overline}

\title{\vspace{-0.in}\parbox{\linewidth }{\footnotesize\noindent
} \\  \bf Approximate controllability of second-order evolution differential inclusions in Hilbert spaces}

\author{\\ N.I. Mahmudov\\ \small {Department of Mathematics, Eastern Mediterranean University, Gazimagusa, }\\ \small{T.R. North Cyprus, Mersin 10, Turkey. email: nazim.mahmudov@emu.edu.tr} \\
V. Vijayakumar\\ \small {Department of Mathematics,  Info Institute of Engineering, Kovilpalayam, }\\  \small{Coimbatore - 641 107, Tamil Nadu, India. email: vijaysarovel@gmail.com}\\
R. Murugesu \\ \small {Department of Mathematics,  SRMV College of Arts and Science,}\\ \small{Coimbatore - 641 020, \ Tamil Nadu, India. \ email: arjhunmurugesh@gmail.com}}
\date{ }

\begin{document}

 \maketitle \footnotetext[1]{ Corresponding author: N.I. Mahmudov}
\author{ \ }

\begin{abstract}
In this paper, we consider a class of second-order evolution differential inclusions in Hilbert spaces. This paper deals with the approximate controllability for a class of second-order control systems. First, we establish a set of sufficient conditions for the approximate controllability for a class of second-order evolution differential inclusions in Hilbert spaces. We use Bohnenblust-Karlin's fixed point theorem to prove our main results. Further, we extend the result to study the approximate controllability concept with nonlocal conditions and extend the result to study the approximate controllability for impulsive control systems with nonlocal conditions. An example is also given to illustrate our main results.
\end{abstract}

{\bf Keywords:} Approximate controllabiliy, Second order differential inclusions, Cosine function of operators, Impulsive systems, Evolution equations, Nonlocal conditions.

{\bf 2010 Mathematics Subject Classification:} 26A33, 34B10, 34K09, 47H10.

\section{Introduction}
\noindent

Controllability is one of the elementary concepts in mathematical control theory. This is a qualitative property of dynamical control systems and it is of particular importance in control theory. Roughly speaking, controllability generally means, that it is possible to steer dynamical control system from an arbitrary initial state to an arbitrary final state using the set of admissible controls. Most of the criteria, which can be met in the literature, are formulated for finite dimensional systems. It should be pointed out that many unsolved problems still exist as far as controllability of infinite dimensional systems are concerned.

In the case of infinite dimensional systems two basic concepts of controllability can be distinguished which are exact and approximate controllability. This is strongly related to the fact that in infinite dimensional spaces there exist linear subspaces, which are not closed. Exact controllability enables to steer the system to arbitrary final state while approximate controllability means that system can be steered to arbitrary small neighborhood of final state. In other words approximate controllability gives the possibility of steering the system to states which form the dense subspace in the state space. In recent years, controllability problems for various types of nonlinear dynamical systems in infinite dimensional spaces by using different kinds of approaches have been considered in many publications, see \cite{na1, ga1, ga2, jpd1, tg1, hrh4, nim1, nim2, nim4, nim5, nim7, nim8, rs6, rs1, rs2, vv5, vv6, zy1, zy2} and the references therein.

Recently, in \cite{nim8}, Mahmudov et al. studied the approximate controllability of second-order neutral stochastic evolution equations using semi-group methods, together with the Banach fixed point theorem. In \cite{rs2} Sakthivel et al. studied the approximate controllability of second-order systems with state-dependent delay by using Schauder fixed point theorem. In \cite{hrh1}, Henr\'{i}quez studied the existence of solutions of non-autonomous second order functional differential equations with infinite delay by using Leray-Schauder's Alternative fixed point theorem. In \cite{zy1}, Yan studied the approximate controllability of fractional neutral integro-differential inclusions with state-dependent delay in Hilbert spaces by using Dhage's fixed point theorem. In \cite{vv6}, Vijayakumar et al. discussed the approximate controllability for a class of fractional neutral integro-differential inclusions with state-dependent delay by using Dhage's fixed point theorem. In \cite{rs6} Sakthivel et al. studied the approximate controllability of fractional nonlinear differential inclusions with initial and nonlocal conditions by using Bohnenblust-Karlin's fixed point theorem. Very recently, in \cite{ga1} Arthi et al. established the sufficient conditiond for controllability of second-order impulsive evolution systems with infinite delay by using Leray-Schauder's fixed point theorem and in \cite{ga2}, proved existence and controllability results for second-order impulsive stochastic evolution systems with state-dependent delay by using Leray-Schauder's fixed point theorem. In \cite{tg1}, Guendouzi investigated the approximate controllability for a class of fractional neutral stochastic functional integro-differential inclusions Bohnenblust-Karlin's fixed point theorem.

Inspired by the above works, in this paper, we establish sufficient conditions for the approximate controllability for a class of second-order evolution differential inclusions in Hilbert spaces of the form
\begin{align}
x''(t)\in&A(t)x(t)+F(t,x(t))+Bu(t), \quad t\in I=[0,b],\label{vi1}\\
x(0)=&x_0\in X,\quad x'(0)=y_0\in X.\label{vi2}
\end{align}
In this equation, $A(t):D(A(t))\subseteq X\ra X$ is a closed linear operator on a Hilbert space $X$ and the control function $u(\cdot)\in L^2(I,U)$, a Hilbert space of admissible control functions. Further, $B$ is a bounded linear operator from $U$ to $X$, and $F:I\times X\ra 2^X\setminus\{\emptyset\}$ is a nonempty, bounded, closed and convex multivalued map.

Converting a second-order system into a first-order system and studying its controllability may not yield desired results due to the behavior of the semigroup generated by the linear part of the converted first order system. So, in many cases, it is advantageous to treat the second-order abstract differential system directly rather than to convert that into first-order system. To the best of our knowledge, the study of the approximate controllability for a class of second-order evolution differential systems in Hilbert spaces treated in this paper, is an untreated topic in the literature, to fill the gap, we study this interesting paper.

This paper is organized as follows. In Section \ref{s3}, we establish a set of sufficient conditions for the approximate controllability for a class of second-order evolution differential inclusions in Hilbert spaces. In Section \ref{s4}, we establish a set of sufficient conditions for the approximate controllability for a class of second-order evolution differential inclusions with nonlocal conditions in Hilbert spaces. In Section \ref{s5}, we establish a set of sufficient conditions for the approximate controllability for a class of second-order impulsive evolution differential inclusions with nonlocal conditions in Hilbert spaces. An example is presented in Section \ref{s6} to illustrate the theory of the obtained results.

\section{Preliminaries}\label{s2}
\noindent

In this section, we mention a few results, notations and lemmas needed to establish our main results. We introduce certain notations which will be used throughout the article without any further mention. Let $(X,\|\cdot\|_X)$ and $(Y,\|\cdot\|_Y)$ be Hilbert spaces, and $\mc{L}(Y,X)$ be the Banach space of bounded linear operators from $Y$ into $X$ equipped with its natural topology; in particular, we use the notation $\mc{L}(X)$ when $Y=X$. By $\rho(A)$, we denote the resolvent set of a linear operator $A$. Throughout this paper, $B_r(x,X)$ will denote the closed ball with center at $x$ and radius $r>0$ in a Hilbert space $X$. We denote by $\mc{C}$, the Banach space $C(J,X)$ endowed with supnorm given by $\|x\|_\mc{C}\equiv \sup_{t\in I}\|x(t)\|$, for $x\in \mc{C}$.

In recent times there has been an increasing interest in studying the abstract non-autonomous second order initial value problem
\begin{align}
x''(t)=&A(t)x(t)+f(t), \quad 0\le s, t \le b,\label{vp1}\\
x(s)=&x_0,\quad x'(s)=y_0,\label{vp2}
\end{align}
where $A(t):D(A(t))\subseteq X\ra X$, $t\in I=[0,b]$ is a closed densely defined operator and $f:I\ra X$ is an appropriate function. Equations of this type have been considered in many papers. The reader is referred to \cite{cjk1, ff1, yp1, yl1, eo1} and the references mentioned in these works. In the most of works, the existence of solutions to the problem $\eqref{vp1}-\eqref{vp2}$ is related to the existence of an evolution operator $S(t,s)$ for the homogeneous equation
\begin{align}
x''(t)=&A(t)x(t), \quad 0\le s, t \le b,\label{vp3}
\end{align}
Let as assume that the domain of $A(t)$ is a subspace $D$ dense in $X$ and independent of $t$, and for each $x\in D$ the function $t\mapsto A(t)x$ is continuous.

Following Kozak \cite{mk1}, in this work we will use the following concept of evolution operator.

\begin{definition}\label{vd1}
A family $S$ of bounded linear operators $S(t,s):I\times I\ra \mc{L}(X)$ is called an evolution operator for $\eqref{vp3}$ if the following conditions are satisfied:
\begin{enumerate}
\item[$(Z1)$] For each $x\in X$, the mapping $[0,b]\times[0,b]\ni(t,s)\ra S(t,s)x\in X$ is of class $C^1$ and
\begin{itemize}
\item[$(i)$] for each $t\in [0,b]$, $S(t,t)=0$,
\item[$(ii)$] for all $t,s\in [0,b]$, and for each $x\in X$,
$$\frac{\partial}{\partial t}S(t,s)x\Big|_{t=s}=x,\qquad \frac{\partial}{\partial t}S(t,s)x\Big|_{t=s}=-x.$$
\end{itemize}
\item[$(Z2)$] For all $t,s\in [0,b]$, if $x\in D(A)$, then $S(t,s)x\in D(A)$, the mapping $[0,b]\times[0,b]\ni(t,s)\ra S(t,s)x\in X$ is of class $C^2$ and
\begin{itemize}
\item[$(i)$] $\frac{\partial^2}{\partial t^2} S(t,x)x=A(t)S(t,s)x$,
\item[$(ii)$] $\frac{\partial^2}{\partial s^2} S(t,x)x=S(t,s)A(s)x$,
\item[$(iii)$] $\frac{\partial}{\partial s}\frac{\partial}{\partial t} S(t,x)x\Big|_{t=s}=0$.
\end{itemize}
\item[$(Z3)$] For all $t,s\in [0,b]$, if $x\in D(A)$, then $\frac{\partial}{\partial s}S(t,s)x\in D(A)$, then $\frac{\partial^2}{\partial t^2}\frac{\partial}{\partial s}S(t,s)x$, $\frac{\partial^2}{\partial s^2}\frac{\partial}{\partial t}S(t,s)x$ and
\begin{itemize}
\item[$(i)$] $\frac{\partial^2}{\partial t^2}\frac{\partial}{\partial s}S(t,s)x=A(t)\frac{\partial}{\partial s}S(t,s)x$,
\item[$(ii)$] $\frac{\partial^2}{\partial s^2}\frac{\partial}{\partial t}S(t,s)x=\frac{\partial}{\partial t}S(t,s)A(s)x$,
\end{itemize}
and the mapping $[0,b]\times[0,b]\ni(t,s)\ra A(t)\frac{\partial}{\partial s}S(t,s)x$ is continuous.
\end{enumerate}
\end{definition}

Throughout this work we assume that there exists an evolution operator $S(t,s)$ associated to the operator $A(t)$. To abbreviate the text, we introduce the operator $C(t,s)=-\frac{\partial S(t,s)}{\partial s}$. In addition, we set $N$ and $\wt{N}$ for positive constants such that $\sup_{0\le t,s \le b}\|S(t,s)\|\le N$ and $\sup_{0\le t,s \le b}\|C(t,s)\|\le \wt{N}$. Furthermore, we denote by $N_1$ a positive constant such that $$\|S(t+h,s)-S(t,s)\|\le N_1|h|,$$
for all $s,t, t+h\in [0,b]$. Assuming that $f:I\ra X$ is an integrable function, the mild solution $x:[0,b]\ra X$ of the problem $\eqref{vp1}-\eqref{vp2}$ is given by
\begin{gather*}
x(t)=C(t,s)x_0+S(t,s)y_0+\int_0^t S(t,\tau)h(\tau)d\tau.
\end{gather*}

In the literature several techniques have been discussed to establish the existence of the evolution operator $S(\cdot,\cdot)$. In particular, a very studied situation is that $A(t)$ is the perturbation of an operator $A$ that generates a cosine operator function. For this reason, below we briefly review some essential properties of the theory of cosine functions. Let $A:D(A)\subseteq X\ra X$ be the infinitesimal generator of a strongly continuous cosine family of bounded linear operators $(C_{0}(t))_{t\in \mb{R}}$ on Hilbert space $X$. We denote by $(S_0(t))_{t\in \mb{R}}$ the sine function associated with $(C_{0}(t))_{t\in \mb{R}}$ which is defined by
\begin{gather*}
S_0(t)x=\int_0^t C(s)xds, \quad x\in X, \quad t\in \mb{R}.
\end{gather*}
We refer the reader to \cite{hof1, cct1, cct2} for the necessary concepts about cosine functions. Next we only mention a few results and notations about this matter needed to establish our results. It is immediate that
\begin{gather*}
C_{0}(t)x-x=A\int_0^t S(s)xds,
\end{gather*}
for all $X$. The notation $[D(A)]$ stands for the domain of the operator $A$ endowed with the graph norm $\|x\|_A = \|x\| + \|Ax\|$, $x \in D(A)$. Moreover, in this paper the notation $E$ stands for the space formed by the vectors $x\in X$ for which the function $C(\cdot)x$ is a class $C^1$ on $\mb{R}$. It was proved by Kisy\'{n}ski \cite{jk1} that the space $E$ endowed with the norm
\begin{gather*}
\|x\|_E=\|x\|+\sup_{0\le t\le 1}\|AS(t,0)x\|, \quad x\in E,
\end{gather*}
is a Hilbert space. The operator valued function
\begin{gather*}
G(t)=\left[\begin{array}{lr} C_{0}(t)& S_0(t)\\ AS_0(t) & C_0(t) \end{array}\right]
\end{gather*}
is a strongly continuous group of linear operators on the space $E\times X$ generated by the operator $ \mc{A}=\left[\begin{array}{lr} 0& I\\ A & 0 \end{array}\right] $ defined on $D(A)\times E$. It follows from this that $AS_0(t):E\ra X$ is a bounded linear operator such that $AS_0(t)x\ra 0$ as $t\ra 0$, for each $x\in E$. Furthermore, if $x:[0,\infty)\ra X$ is a locally integrable function, then $z(t)=\int_0^t S(t,s)x(s)ds$ defines an $E$-valued continuous function.

The existence of solutions for the second order abstract Cauchy problem
\begin{align}
x''(t)=&Ax(t)+h(t), \quad 0\le t\le b,\label{vp4}\\
x(0)=&x_0, \quad x'(0)=y_0,\label{vp5}
\end{align}
where $h:[0,b]\ra X$ is an integrable function, has been discussed in \cite{cct1}. Similarly, the existence of solutions of the semilinear second order Cauchy problem it has been treated in \cite{cct2}. We only mention here that the function $x(\cdot)$ given by
\begin{gather}
x(t)=C_0(t-s)x_0+S_0(t-s)y_0+\int_s^t S_0(t-\tau)h(\tau)d\tau, \quad 0\le t \le b, \label{vp6}
\end{gather}
is called the mild solution of $\eqref{vp4}-\eqref{vp5}$ and that when $x_0\in E$, $x(\cdot)$ is continuously differentiable and
\begin{gather*}
x'(t)=AS_0(t-s)x_0+C_0(t-s)y_0+\int_s^t C_0(t-\tau)h(\tau)d\tau, \quad 0\le t \le b.
\end{gather*}
In addition, if $x_0\in D(A)$, $y_0\in E$ and $f$ is a continuously differentiable function, then the function $x(\cdot)$ is a solution of the initial value problem $\eqref{vp4}-\eqref{vp5}$.

Assume now that $A(t)=A+\wt{B}(t)$ where $\wt{B}(\cdot):\mb{R}\ra \mc{L}(E,X)$ is a map such that the function $t\ra \wt{B}(t)x$ is continuously differentiable in $X$ for each $x\in E$. It has been established by Serizawa \cite{hs1} that for each $(x_0,y_0)\in D(A)\times E$ the nonautonomous abstract Cauchy problem
\begin{align}
x''(t)=&(A+\wt{B}(t))x(t), \quad t\in \mb{R},\label{vp7}\\
x(0)=&x_0, \quad x'(0)=y_0,\label{vp8}
\end{align}
has a unique solution $x(\cdot)$ such that the function $t\mapsto x(t)$ is continuously differentiable in $E$. It is clear that the same argument allows us to conclude that equation \eqref{vp7} with the initial condition \eqref{vp5} has a unique solution $x(\cdot,s)$ such that the function $t\mapsto x(t,s)$ is continuously differentiable in $E$. It follows from \eqref{vp6} that
\begin{gather*}
x(t,s)=C_0(t-s)x_0+S_0(t-s)y_0+\int_s^t S(t-\tau)\wt{B}(\tau)x(\tau,s)d\tau.
\end{gather*}
In particular, for $x_0=0$ we have
\begin{gather*}
x(t,s)=S_0(t-s)y_0+\int_s^t S(t-\tau)\wt{B}(\tau)x(\tau,s)d\tau.
\end{gather*}
Consequently,
\begin{gather*}
\|x(t,s)\|_1\le \|S_0(t-s)\|_{\mc{L}(X,E)}\|y_0\|+\int_s^t \|S_0(t-s)\|_{\mc{L}(X,E)}\|\wt{B}(\tau)\|_{\mc{L}(X,E)}\|x(s,\tau)\|_1d\tau.
\end{gather*}
and, applying the Gronwall-Bellman lemma we infer that
\begin{gather*}
\|x(t,s)\|_1\le \wt{M}\|y_0\|,\quad s, t\in I.
\end{gather*}
We define the operator $S(t,s)y_0=x(t,s)$. It follows from the previous estimate that $S(t,s)$ is a bounded linear map on $E$. Since $E$ is dense in $X$, we can extend $S(t,s)$ to $X$. We keep the notation $S(t,s)$ for this extension. It is well known that, except in the case $\dim(X)<\infty$, the cosine function $C_0(t)$ cannot be compact for all $t\in \mb{R}$. By contrast, for the cosine functions that arise in specific applications, the sine function $S_0(t)$ is very often a compact operator for all $t\in \mb{R}$.

\begin{theorem}\cite[Theorem 1.2]{hrh1}.
Under the preceding conditions, $S(\cdot,\cdot)$ is an evolution operator for $\eqref{vp7}-\eqref{vp8}$. Moreover, if $S_0(t)$ is compact for all $t\in \mb{R}$, then $S(t,s)$ is also compact for all $s\le t$.
\end{theorem}

We also introduce some basic definitions and results of multivalued maps. For more details on multivalued maps, see the books of Deimling \cite{kde1} and Hu and Papageorgious \cite{sh1}.

A multivalued map $G: X\ra 2^X\setminus \{\emptyset\}$ is convex (closed) valued if $G(x)$ is convex (closed) for all $x\in X$. $G$ is bounded on bounded sets if $G(C)=\bigcup_{x\in C} G(x)$ is bounded in $X$ for any bounded set $C$ of $X$, i.e., $\sup_{x\in C}\Big\{\sup\{\|y\|:y\in G(x)\}\Big\}<\infty$.

\begin{definition}
$G$ is called upper semicontinuous (u.s.c. for short) on $X$ if for each $x_0\in X$, the set $G(x_0)$ is a nonempty closed subset of $X$, and if for each open set $C$ of $X$ containing $G(x_0)$, there exists an open neighborhood $V$ of $x_0$ such that $G(V)\subseteq C$.
\end{definition}

\begin{definition}
$G$ is called completely continuous if $G(C)$ is relatively compact for every bounded subset $C$ of $X$.
\end{definition}

If the multivalued map $G$ is completely continuous with nonempty values, then $G$ is u.s.c., if and only if $G$ has a closed graph, i.e., $x_n\ra x_*$, $y_n\ra y_*$, $y_n\in Gx_n$ imply $y_*\in Gx_*$. $G$ has a fixed point if there is a $x\in X$ such that $x\in G(x)$.

\begin{definition}
A function $x\in \mc{C}$ is said to be a mild solution of system \eqref{vi1}-\eqref{vi2} if $x(0)=x_0$, $x'(0)=y_0$ and there exists $f\in L^1(I,X)$ such that $f(t)\in F(t,x(t))$ on $t\in I$ and the integral equation
\begin{align*}
x(t)=C(t,0)x_0+S(t,0)y_0+\int_0^tS(t,s)f(s)ds+\int_0^tS(t,s)Bu(s)ds,\  t\in I.
\end{align*}
is satisfied.
\end{definition}

In order to address the problem, it is convenient at this point to introduce two relevant operators and basic assumptions on these operators:

\begin{align*}
\Upsilon_0^b&=\int_0^bS(b,s)BB^*S^*(b,s)ds: X\ra X,\\
R(a,\Upsilon_0^b)&=(a I+\Upsilon_0^b)^{-1}: X\ra X.
\end{align*}
It is straightforward that the operator $\Upsilon_0^b$ is a linear bounded operator.

To investigate the approximate controllability of system \eqref{vi1}-\eqref{vi2}, we impose the following condition:
\begin{enumerate}
\item [$\bf H_0$] $aR(a,\Upsilon_0^b)\ra 0$ as $a\ra 0^+$ in the strong operator topology.
\end{enumerate}

In view of \cite{nim1}, Hypothesis $\bf{H_0}$ holds if and only if the linear system
\begin{align}
x''(t)=&Ax(t)+(Bu)(t),\quad t\in [0,b],  \label{s2e1} \\
x(0)=&x_0 \quad x'(0)=y_0,\label{s2e2}
\end{align}
is approximately controllable on $[0,b]$.

Some of our results are proved using the next well-known results.

\begin{lemma}\cite[Lasota and Opial]{al1}
Let $I$ be a compact real interval, $BCC(X)$ be the set of all nonempty, bounded, closed and convex subset of $X$ and $F$ be a multivalued map satisfying $F:I\times X\ra BCC(X)$ is measurable to $t$ for each fixed $x\in X$, u.s.c. to $x$ for each $t\in I$, and for each $x\in \mc{C}$ the set
\begin{align*}
S_{F,x}=\{f\in L^1(I,X):f(t)\in F(t,x(t)),\  t\in I\}
\end{align*}
is nonempty. Let $\ms{F}$ be a linear continuous from $L^1(I,X)$ to $\mc{C}$, then the operator
\begin{gather*}
\ms{F} \circ S_F:\mc{C}\ra BCC(\mc{C}),\ x\ra (\ms{F} \circ S_F)(x)=\ms{F}(S_{F,x}),
\end{gather*}
is a closed graph operator in $\mc{C}\times \mc{C}$.
\end{lemma}

\begin{lemma} \cite[Bohnenblust and Karlin]{hfb1}. \label{29}
Let $\mc{D}$ be a nonempty subset of $X$, which is bounded, closed, and convex. Suppose $G:\mc{D}\ra 2^X\setminus \{\emptyset\}$ is u.s.c. with closed, convex values, and such that $G(\mc{D})\subseteq \mc{D}$ and $G(\mc{D})$ is compact. Then $G$ has a fixed point.
\end{lemma}

\section{Approximate controllability results}\label{s3}
\noindent

In this section, first we establish a set of sufficient conditions for the approximate controllability for a class of second order evolution differential inclusions of the form \eqref{vi1}-\eqref{vi2} in Hilbert spaces by using Bohnenblust-Karlin's fixed point theorem. In order to establish the result, we need the following hypotheses:

\begin{enumerate}
\item [$\bf H_1$] $S_0(t)$, $t>0$ is compact.

\item[$\bf H_2$] For each positive number $r$ and $x\in \mc{C}$ with $\|x\|_\mc{C}\le r$, there exists $L_{f,r}(\cdot)\in L^1(I,\mb{R}^+)$ such
that
\begin{gather*}
\sup\{\|f\|:f(t)\in F(t,x(t))\}\le L_{f,r}(t),
\end{gather*}
for a.e. $t\in I$.

\item[$\bf H_3$] The function $s\ra  L_{f,r}(s)\in L^1([0,t],\mb{R}^+)$ and there exists a $\gamma>0$ such that
\begin{gather*}
\lim_{r\ra \infty} \frac{\int_0^t L_{f,r}(s)ds}{r}=\gamma <+\infty.
\end{gather*}
\end{enumerate}

It will be shown that the system \eqref{vi1}-\eqref{vi2} is approximately controllable, if for all $a>0$, there exists a continuous function $x(\cdot)$ such that
\begin{align}
x(t)=&C(t,0)x_0+S(t,0)y_0+\int_0^t S(t,s)f(s)ds+\int_0^t S(t,s)Bu(s,x)ds,\quad f\in S_{F,x},\label{c7m1}\\
u(t,x)&=B^*S(b,t)R(a,\Upsilon_0^b)p(x(\cdot)),\label{c7m2}
\end{align}
where $$ p(x(\cdot))=x_b-C(b,0)x_0-S(b,0)y_0-\int_0^t S(b,s)f(s)ds.$$

\begin{theorem}\label{34}
Suppose that the hypotheses $\bf{H_0}$-$\bf{H_3}$ are satisfied. Assume also
\begin{align}
\wt{N}\gamma\Big[1+\frac{1}{\alpha}\wt{N}^2M^2_{B}b\Big]<1,\label{31}
\end{align}
where $M_B=\|B\|$, Then system \eqref{vi1}-\eqref{vi2} has a solution on $I$.
\end{theorem}

\begin{proof}
The main aim in this section is to find conditions for solvability of system \eqref{vi1}-\eqref{vi2} for $a>0$. We show that, using the control $u(x,t)$, the operator ${\Gamma}:\mc{C}\ra 2^{\mc{C}}$, defined by
\begin{align*}
{\Gamma}(x)=&\Big\{\varphi\in \mc{C}:\varphi (t)=C(t,0)x_0+S(t,0)y_0+\int_0^t S(t,s)[f(s)+Bu(s,x)]ds,\ f\in S_{F,x}\Big\},
\end{align*}
has a fixed point $x$, which is a mild solution of system \eqref{vi1}-\eqref{vi2}.

We now show that $\Gamma$ satisfies all the conditions of Lemma \ref{29}. For the sake of convenience, we subdivide the proof into five steps.

\noindent{\bf Step 1.} $\Gamma$ is convex for each $x\in \mc{C}$.

In fact, if $\varphi_1$, $\varphi_2$ belong to $\Gamma(x)$, then there exist $f_1$, $f_2\in S_{F,x}$ such that for each $t\in I$, we have
\begin{align*}
\varphi_i(t)=&C(t,0)x_0+S(t,0)y_0+\int_0^t S(t,s)f_i(s)ds+\int_0^t S(t,s)BB^{*}S^{*}(b,t) R(a,\Upsilon_0^b)\Big[x_b-C(b,0)x_0\\&-S(b,0)y_0-\int_0^b  S(b,\eta)f_i(\eta)d\eta\Big](s) ds,\quad i=1,2.
\end{align*}
Let $\lambda\in [0,1]$. Then for each $t\in I$, we get
\begin{align*}
\lambda\varphi_1(t)+(1-\lambda)\varphi_2(t)=&C(t,0)x_0+S(t,0)y_0+\int_0^t S(t,s)[\lambda f_1(s) +(1-\lambda)f_2(s)]ds\\&+\int_0^t S(t,s)BB^{*}S^{*}(b,t) R(a,\Upsilon_0^b)\Bigg[x_b-C(b,0)x_0-S(b,0)y_0\\&-\int_0^b S(b,s)[\lambda f_1(s) +(1-\lambda)f_2(s)]ds\Bigg](s)ds.
\end{align*}
It is easy to see that $S_{F,x}$ is convex since $F$ has convex values. So, $\lambda f_1+(1-\lambda)f_2\in S_{F,x}$. Thus,
$$\lambda\varphi_1+(1-\lambda)\varphi_2\in \Gamma(x).$$

\noindent{\bf Step 2.} For each positive number $r>0$, let $\mf{B}_r=\{x\in \mc{C}:\|x\|_\mc{C}\le r\}$. Obviously, $\mf{B}_r$ is a bounded, closed and convex set of $\mc{C}$. We claim that there exists a positive number $r$ such that $\Gamma(\mf{B}_r )\subseteq \mf{B}_r$.

If this is not true, then for each positive number $r$, there exists a function $x^r\in \mf{B}_r$, but $\Gamma(x^r)$ does not belong to $\mf{B}_r$, i.e.,
$$\|\Gamma(x^r)\|_\mc{C}\equiv\sup\Big\{\|\varphi^r\|_\mc{C}:\varphi^r\in (\Gamma x^r)\Big\}>r$$
and
\begin{align*}
\varphi^r(t)=&C(t,0)x_0+S(t,0)y_0+\int_0^t S(t,s)f^r(s)ds+\int_0^t S(t,s)Bu^r(s,x) ds,
\end{align*}
for some $f^r\in S_{F,x^r}$. Using $\bf H_1$-$\bf H_3$, we have
\begin{align*}
r<&\|\Gamma(x^r)(t)\|\\
\le&\|C(t,0)x_0\|+\|S(t,0)y_0\|+\int_0^t \|S(t,s)f^r(s)\|ds+\int_0^t \|S(t,s)Bu^r(s,x)\|ds\\
\le& \Big[N\|x_0\|+\wt{N}\|y_0\|+\wt{N}\int_0^t L_{f,r}(s)ds\Big]\\& +\frac{1}{\alpha}\wt{N}^2M^2_{B}b\Bigg[\|x_b\|+N\|x_0\|+\wt{N}\|y_0\|+\wt{N}\int_0^b L_{f,r}(s)ds\Bigg].
\end{align*}
Dividing both sides of the above inequality by $r$ and taking the limit as $r\ra \infty$, using $\bf H_3$, we get
\begin{gather*}
\wt{N}\gamma\Big[1+\frac{1}{\alpha}\wt{N}^2M^2_{B}b\Big]\ge 1.
\end{gather*}
This contradicts with the condition \eqref{31}. Hence, for some $r>0$, $\Gamma(\mf{B}_r )\subseteq \mf{B}_r$.

\noindent{\bf Step 3.} $\Gamma$ sends bounded sets into equicontinuous sets of $\mc{C}$. For each $x\in \mf{B}_r$, $\varphi\in\Gamma(x)$, there exists a $f\in S_{F,x}$ such that
\begin{align*}
\varphi(t)=&C(t,0)x_0+S(t,0)y_0+\int_0^t S(t,s)f(s)ds+\int_0^t S(t,s)Bu(s,x)ds.
\end{align*}
Let $0<\ve<0$ and $0<t_1<t_2\le b$, then
\begin{align*}
|\varphi(t_1)-\varphi(t_2)|=&|C(t_1,0)-C(t_2,0)||x_0|+|S(t_1,0)-S(t_2,0)||\eta|\\&+\Big|\int_{0}^{t_1-\ve}[S(t_1,s)-S(t_2,s)]f(s) ds\Big| +\Big|\int_{t-\ve}^{t_1}[S(t_1,s)-S(t_2,s)]f(s) ds\Big|\\&+\Big|\int_{t_1}^{t_2}S(t_2,s)f(s)ds\Big| +\Big|\int_{0}^{t_1-\ve}[S(t_1,\eta)-S(t_2,\eta)]Bu(\eta,x) d\eta\Big|\\& +\Big|\int_{t-\ve}^{t_1}[S(t_1,\eta)-S(t_2,\eta)]Bu(\eta,x) d\eta\Big|+\Big|\int_{t_1}^{t_2}S(t_2,\eta)Bu(\eta,x) d\eta\Big|\\&+\Big|\int_{0}^{t_1-\ve}[S(t_1,\eta)-S(t_2,\eta)]Bu(\eta,x) d\eta\Big|\\
\le &|C(t_1,0)-C(t_2,0)||x_0|+|S(t_1,0)-S(t_2,0)||\eta|\\&+\int_{0}^{t_1-\ve}|S(t_1,s)-S(t_2,s)|L_{f,r}(s) ds\\& +\int_{t-\ve}^{t_1}|S(t_1,s)-S(t_2,s)|L_{f,r}(s) ds+\wt{N}\int_{t_1}^{t_2}L_{f,r}(s) ds\\& +M_B\int_{0}^{t_1-\ve}|S(t_1,\eta)-S(t_2,\eta)|\|u(\eta,x)\| d\eta\\& +M_B\int_{t-\ve}^{t_1}|S(t_1,\eta)-S(t_2,\eta)\|u(\eta,x)\| d\eta+\wt{N}M_B\int_{t_1}^{t_2}\|u(\eta,x)\| d\eta.
\end{align*}

The right-hand side of the above inequality tends to zero independently of $x\in B_r$ as $(t_1-t_2)\ra 0$ and $\ve$ sufficiently small, since the compactness of the evolution operator $S(t,s)$ implies the continuity in the uniform operator topology. Thus $\Gamma(x^r)$ sends $B_r$ into equicontinuous family of functions.

\noindent{\bf Step 4.} The set $\Pi(t)=\big\{\varphi(t):\varphi\in \Gamma(\mf{B}_r)\big\}$ is relatively compact in $X$.

Let $t\in (0,b]$ be fixed and $\ve$ a real number satisfying $0< \ve<t$. For $x\in B_r$, we define
\begin{align*}
\varphi_\ve(t)=C(t,0)x_0+S(t,0)y_0+\int_0^{t-\ve}S(t,s)f(s)ds+\int_0^{t-\ve}S(t,\eta)Bu(\eta,x)d\eta.
\end{align*}
Since $S_0(t)$ is a compact operator, the set $\Pi_\ve(t)=\{\varphi_\ve(t):\varphi_\ve\in \Gamma(B_r)\}$ is relatively compact in $X$ for each $\ve$, $0<\ve<t$. Moreover, for each $0<\ve<t$, we have
\begin{align*}
|\varphi(t)-\varphi_\ve(t)|\le \wt{N}\int_{t-\ve}^tL_{f,r}(s)ds+\wt{N}M_B\int_{t-\ve}^t\|u(\eta,x)\|d\eta.
\end{align*}
Hence there exist relatively compact sets arbitrarily close to the set $\Pi(t)=\{\varphi(t):\varphi\in \Gamma(B_r)\}$, and the set $\wt{\Pi}(t)$ is relatively compact in $X$ for all $t\in [0,b]$. Since it is compact at $t=0$, hence $\Pi(t)$ is relatively compact in $X$ for all $t\in [0,b]$.

\noindent{\bf Step 5.} $\Gamma$ has a closed graph.

Let $x_n\ra x_*$ as $n\ra \infty$, $\varphi_n\in {\Gamma}(x_n)$, and $\varphi_n\ra \varphi_*$ as $n\ra \infty$. We will show that $\varphi_*\in {\Gamma}(x_*)$. Since $\varphi_n\in {\Gamma}(x_n)$, there exists a $f_n\in S_{F,x_n}$ such that
\begin{align*}
\varphi_n(t)=&C(t,0)x_0+S(t,0)y_0+\int_0^tS(t,s)f_n(s)ds+\int_0^t S(t,s)BB^{*}S^{*}(b,t)
R(a,\Upsilon_0^b)\Big[x_b\\&-S(b,0)x_0-\int_0^b  S(b,\eta)f_n(\eta)d\eta\Big](s)ds.
\end{align*}

We must prove that there exists a $f_*\in S_{F,x_*}$ such that
\begin{align*}
\varphi_*(t)=&C(t,0)x_0+S(t,0)y_0+\int_0^t S(t,s)f_*(s)ds+\int_0^t S(t,s)BB^{*}S^{*}(b,t)\Big[y_0-S(b,0)x_0\\&-\int_0^b S(b,\eta)f_*(\eta)ds\Big](s)ds.
\end{align*}
Set
\begin{gather*}
\ov{u}_x(t)=B^{*}S^{*}(b,t)[x_b-C(b,0)x_0-S(b,0)y_0](t).
\end{gather*}
Then
\begin{gather*}
\ov{u}_{x_n}(t)\ra \ov{u}_{x_*}(t), \quad \mbox{for} \ t\in I, \ \mbox{as} \ n\ra \infty.
\end{gather*}
Clearly, we have
\begin{align*}
\Big\|\Big(\varphi_n&-C(t,0)x_0+S(t,0)y_0-\int_0^t S(t,s)BB^{*}S^{*}(b,t)
R(a,\Upsilon_0^b)\Big[x_b-S(b,0)x_0 \\&-\int_0^b  S(b,\eta)f_n(\eta)d\eta\Big](s)ds\Big)-\Big(\varphi_*-C(t,0)x_0+S(t,0)y_0-\int_0^t S(t,s)BB^{*}S^{*}(b,t)\Big[y_0\\&-C(b,0)x_0-S(b,0)y_0-\int_0^b S(b,\eta)f_*(\eta)ds\Big](s)ds\Big)\Big\|_\mc{C} \ra 0 \ \mbox{as} \ n\ra \infty.
\end{align*}
Consider the operator $\wt{\ms{F}}:L^1(I,X)\ra \mc{C}$,
\begin{align*}
(\wt{\ms{F}}f)(t)=\int_0^t S(t,s)\Big[f(s)-BB^{*}S^{*}(b,t)\Big(\int_0^b S(b,\eta)f(\eta)d\eta\Big)(s)\Big]ds
\end{align*}

We can see that the operator $\wt{\ms{F}}$ is linear and continuous. From Lemma \ref{29} again, it follows that $\wt{\ms{F}}\circ S_F$ is a closed graph operator. Moreover,
\begin{align*}
\Big(\varphi_n&-C(t,0)x_0-S(t,0)y_0-\int_0^t S(t,s)BB^{*}S^{*}(b,t)
R(a,\Upsilon_0^b)\Big[x_b-S(b,0)x_0 \\&-\int_0^b S(b,\eta)f_n(\eta)d\eta\Big](s)ds\Big)\in \ms{F}(S_{F,x_n}).
\end{align*}

In view of $x_n\ra x_*$ as $n\ra \infty$, it follows again from Lemma \ref{29} that
\begin{align*}
\Big(\varphi_*&-C(t,0)x_0+S(t,0)y_0-\int_0^t S(t,s)BB^{*}S^{*}(b,t)
R(a,\Upsilon_0^b)\Big[x_b-S(b,0)x_0 \\&-\int_0^b S(b,\eta)f_*(\eta)d\eta\Big](s)ds\Big)\in \ms{F}(S_{F,x_*}).
\end{align*}
Therefore $\Gamma$ has a closed graph.

As a consequence of {\bf Steps 1-5} together with the Arzela-Ascoli theorem, we conclude that $\Gamma$ is a compact multivalued map, u.s.c. with convex closed values. As a consequence of Lemma \ref{29}, we can deduce that $\Gamma$ has a fixed point $x$ which is a mild solution of system \eqref{vi1}-\eqref{vi2}.
\end{proof}

\begin{definition}
The control system \eqref{vi1}-\eqref{vi2} is said to be approximately controllable on $I$ if $\ov{R(b)}=X$, where $R(b)=\{x(b;u):u^2L(I,U)\}$ is a mild solution of the system \eqref{vi1}-\eqref{vi2}.
\end{definition}

\begin{theorem}\label{40}
Suppose that the assumptions $\bf {H_0}$-$\bf {H_3}$ hold. Assume additionally that there exists $N\in L^1(I,[0,\infty))$ such that $\sup_{x\in X}\|F(t,x)\|\le N(t)$ for a.e. $t\in I$, then the nonlinear fractional differential inclusion \eqref{vi1}-\eqref{vi2} is approximately controllable on $I$.
\end{theorem}

\begin{proof}
Let $\widehat{x}^a(\cdot)$ be a fixed point of $\Gamma$ in $\mf{B}_r$. By Theorem \ref{34}, any fixed point of $\Gamma$ is a mild solution of \eqref{vi1}-\eqref{vi2} under the control
\begin{align*}
\widehat{u}^a(t)=B^*S^*(b,t)R(a,\Upsilon_0^b)p(\widehat{x}^a)
\end{align*}
and satisfies the following inequality
\begin{align*}
\widehat{x}^a(b)=x_b+aR(a,\Upsilon_0^b)p(\widehat{x}^a).\label{kd1}
\end{align*}
Moreover by assumption on $F$ and Dunford-Pettis Theorem, we have that the $\{f^a(s)\}$ is weakly compact in $L^1(I,X)$, so there is a subsequence, still denoted by $\{f^a(s)\}$, that converges weakly to say $f(s)$ in $L^1(I,X)$. Define
\begin{gather*}
w=x_b-C(b,0)x_0-S(b,0)y_0-\int_0^b S(b,s)f(s)ds.
\end{gather*}
Now, we have
\begin{align}
\n \|p(\widehat{x}^a)-w\|=&\Big\|\int_0^b S(b,s)[f(s,\widehat{x}^a(s))-f(s)]ds\Big\|\\
\le&\sup_{t\in I}\Big\|\int_0^t  S(t,s)[f(s,\widehat{x}^a(s))-f(s)]ds\Big\|.\label{kd2}
\end{align}
By using infinite-dimensional version of the Ascoli-Arzela theorem, one can show that an operator $l(\cdot)\ra \int_0^{\cdot} S(\cdot,s)l(s)ds:L^1(I,X) \ra \mc{C}$ is compact. Therefore, we obtain that $\|p(\widehat{x}^a)-w\|\ra 0$ as $a\ra 0^+$. Moreover, from \eqref{kd1} we get
\begin{align*}
\|\widehat{x}^a(b)-x_b\|\le & \|aR(a,\Upsilon_0^b)(w)\|+\|aR(a,\Upsilon_0^b)\|\|p(\widehat{x}^a)-w\|\\
\le&\|aR(a,\Upsilon_0^b)(w)\|+\|p(\widehat{x}^a)-w\|.
\end{align*}
It follows from assumption $\bf{H_0}$ and the estimation \eqref{kd2} that $\|\widehat{x}^a(b)-x_b\|\ra 0$ as $a\ra 0^+$. This proves the approximate controllability of differential inclusion \eqref{vi1}-\eqref{vi2}.
\end{proof}

\section{Second-order control systems with nonlocal conditions} \label{s4}
\noindent

There exist an extensive literature of  differential equations with nonlocal conditions. Since it is demonstrated that the nonlocal problems have better effects in applications than the classical ones, differential equations with nonlocal problems have been studied extensively in the literature. The result concerning the existence and uniqueness of mild solutions to abstract Cauchy problems with nonlocal initial conditions was first formulated and proved by Byszewski, see \cite{lb1, lb2, lb3}. Since the appearance of these papers, several papers have addressed the issue of existence and uniqueness of nonlinear differential equations. Existence and controllability results of nonlinear differential equations and fractional differential equations with nonlocal conditions has been studied by several authors for different kind of problems \cite{sa1, sa2, hrh3, ke1, xf1, jl1, nim2, vv5, xx1}.

Inspired by the above works, in this section, we discuss the approximate controllability for a class of second-order evolution differential inclusions with nonlocal conditions in Hilbert spaces of the form
\begin{align}
x''(t)\in&A(t)x(t)+F(t,x(t))+Bu(t), \quad t\in I=[0,b],\label{vn1}\\
x(0)+&g(x)=x_0\in X,\quad x'(0)+h(x)=y_0\in X,\label{vn2}
\end{align}
where $g,h:\mc{C}\ra X$ is a given function which satisfies the following condition:

\begin{enumerate}
\item[$\bf H_4$] There exists a constant $L_g>0$, $L_h>0$ such that
\begin{align*}
|g(x)-g(y)|\le& L_g\|x-y\|,\ \text{for}\ x,y\in \mc{C}\\
|h(x)-h(y)|\le& L_h\|x-y\|,\ \text{for}\ x,y\in \mc{C}.
\end{align*}
for all $x,y\in \mc{C}$.
\end{enumerate}

\begin{definition}
A function $x\in \mc{C}$ is said to be a mild solution of system \eqref{vn1}-\eqref{vn2} if $x(0)+g(x)=x_0$, $x'(0)+h(x)=y_0$ and there exists $f\in L^1(I,X)$ such that $f(t)\in F(t,x(t))$ on $t\in I$ and the integral equation
\begin{align*}
x(t)=C(t,0)(x_0-g(x))+S(t,0)(y_0-h(x))+\int_0^tS(t,s)f(s)ds+\int_0^tS(t,s)Bu(s)ds,\  t\in I.
\end{align*}
is satisfied.
\end{definition}

\begin{theorem}\label{3.4}
Assume that the assumptions of Theorem \ref{34} are satisfied. Further, if the hypothesis $\bf H_4$ is satisfied, then the system \eqref{vn1}-\eqref{vn2} is approximately controllable on $I$ provided that
\begin{align*}
\wt{N}\gamma\Big[1+\frac{1}{\alpha}\wt{N}^2M^2_{B}b\Big]<1
\end{align*}
where $M_B=\|B\|$.
\end{theorem}

\begin{proof}
For each $a>0$, we define the operator $\widehat{\Gamma}_a$ on $X$ by
\begin{align*}
(\widehat{\Gamma}_a x)=z,
\end{align*}
where
\begin{align*}
z(t)&=C(t,0)(x_0-g(x))+S(t,0)(y_0-h(x))+\int_0^t S(t,s)f(s)ds+\int_0^t S(t,s)Bu(s,x)ds,\\
v(t)&=B^*S^*(b,t)R(a,\Upsilon_0^b)p(x(\cdot)),\\
p(x(\cdot))&=x_b-C(t,0)(x_0-g(x))-S(t,0)(y_0-g(x))-\int_0^t S(t,s)f(s)ds
\end{align*}
where $ f\in S_{F,x}$.

It can be easily proved that if for all $a>0$, the operator $\widehat{\Gamma}_a$ has a fixed point by implementing the technique used in Theorem \ref{34}. Then, we can show that the second order control system \eqref{vn1}-\eqref{vn2} is approximately controllable. The proof of this theorem is similar to that of Theorem \ref{34} and Theorem \ref{40}, and hence it is omitted.
\end{proof}

\section{Second-order impulsive nonlocal control systems} \label{s5}
\noindent

Impulsive dynamical systems are characterized by the occurrence of an abrupt change in the state of the system, which occur at certain time instants over a period of negligible duration. The dynamical behavior of such systems is much more complex than the behavior of dynamical systems without impulse effects. The presence of impulse means that the state trajectory does not preserve the basic properties which are associated with non-impulsive dynamical systems. In fact, the
theory of impulsive differential equations has found extensive applications in realistic mathematical modeling of a wide variety of practical situations \cite{bmm1} and has emerged as an important area of investigation in recent years. It is known that many biological phenomena involving thresholds, bursting rhythm models in medicine and biology, optimal control models in economics, pharmacokinetics and frequency modulation systems, do exhibit impulse effects; see the monographs of Samoilenko and Perestyuk \cite{ams1}, Lakshmikantham et al. \cite{vl1}, Bainov et al. \cite{ddb1} and Benchohra et al. \cite{mb1} and the references cited therein. The literature on this type of problem is vast, and different topics on the existence and qualitative properties of solutions are considered. Concerning the general motivations, relevant developments and the current status of the theory, we refer the reader to \cite{na1, ga1, ga2, eh1, hrh2, hrh3, mk1, yp1, rs1, rs2, ss1, vv3, vv1}.

Inspired by the above works, in this section, we discuss the approximate controllability for a class of second-order impulsive evolution differential inclusions with nonlocal conditions in Hilbert spaces of the form
\begin{align}
x''(t)\in&A(t)x(t)+F(t,x(t))+Bu(t), \quad t\in I=[0,b],\label{vni1}\\
x(0)+&g(x)=x_0\in X,\quad x'(0)+h(x)=y_0\in X,\label{vni2}\\
\Delta x(t_i)=&I_i(x(t_i)),\label{vni3}\\
\Delta x'(t_i)=&J_i(x(t_i)),\quad i=1,2,\cdots,n,\label{vni4}
\end{align}
where $0<t_1<t_2<\cdots<t_n<b$ are fixed numbers and $I_i:X\ra X$, $J_i:X\ra X$, $i=1,...,n$, are suitable functions and the symbol $\Delta \xi(t)$ represents the jump of the function $\xi$ at $t$, which is defined by $\Delta \xi(t)=\xi(t^+)-\xi(t^-)$.

Consider the following assumption
\begin{enumerate}
\item[$\bf H_5$] The maps $I_i:X\ra X$, $J_i:X\ra X$ are completely continuous and uniformly bounded. In the sequel, we set $M_i=\sup\{\|I_i(x)\|:x\in X\}$ and $\wt{M}_i=\sup\{\|J_i(x)\|:x\in X\}$.
\end{enumerate}

\begin{definition}
A function $x\in \mc{C}$ is said to be a mild solution of system \eqref{vni1}-\eqref{vni4} if $x(0)+g(x)=x_0$, $x'(0)+h(x)=y_0$ and there exists $f\in L^1(I,X)$ such that $f(t)\in F(t,x(t))$ on $t\in I$ and the integral equation
\begin{align*}
x(t)=&C(t,0)(x_0-g(x))+S(t,0)(y_0-h(x))+\int_0^tS(t,s)f(s)ds+\int_0^tS(t,s)Bu(s)ds\\&+\sum_{0<t_i<t} C(t,t_i)I_i(u(t_i))+\sum_{0<t_i<t} S(t,t_i)J_i(u(t_i)),\quad  t\in I.
\end{align*}
is satisfied.
\end{definition}

\begin{theorem}
Assume that the assumptions of Theorem \ref{34} and Theorem \ref{3.4} are satisfied. Further, if the hypothesis $\bf H_5$ is satisfied, then the system \eqref{vni1}-\eqref{vni4} is approximately controllable on $I$ provided that
\begin{align*}
\wt{N}\gamma\Big[1+\frac{1}{\alpha}\wt{N}^2M^2_{B}b\Big]<1
\end{align*}
where $M_B=\|B\|$.
\end{theorem}

\begin{proof}
For each $a>0$, we define the operator $\widehat{\Gamma}_a$ on $X$ by
\begin{align*}
(\widehat{\Gamma}_a x)=z,
\end{align*}
where
\begin{align*}
z(t)=&C(t,0)(x_0-g(x))+S(t,0)(y_0-h(x))+\int_0^t S(t,s)f(s)ds+\int_0^t S(t,s)Bu(s,x)ds\\&+\sum_{0<t_i<t} C(t,t_i)I_i(u(t_i))+\sum_{0<t_i<t} S(t,t_i)J_i(u(t_i)),\\
v(t)=&B^*S^*(b,t)R(a,\Upsilon_0^b)p(x(\cdot)),\\
p(x(\cdot))=&x_b-C(t,0)(x_0-g(x))-S(t,0)(y_0-g(x))-\int_0^t S(t,s)f(s)ds\\&+\sum_{i=1}^n C(t,t_i)I_i(u(t_i))+\sum_{i=1}^n S(t,t_i)J_i(u(t_i)),
\end{align*}
where $ f\in S_{F,x}$.

It can be easily proved that if for all $a>0$, the operator $\widehat{\Gamma}_a$ has a fixed point by implementing the technique used in Theorem \ref{34}. Then, we can show that the second order control system \eqref{vni1}-\eqref{vni4} is approximately controllable. The proof of this theorem is similar to that of Theorem \ref{34}, Theorem \ref{40} and Theorem \ref{3.4}, and hence it is omitted.
\end{proof}

\section{An application}\label{s6}
\noindent

In this section, we apply our abstract results to a concrete partial differential equation. In order to establish our results, we need to introduce the required technical tools. Following the equations $\eqref{vp7}-\eqref{vp8}$, here we consider $A(t)=A+\wt{B}(t)$ where $A$ is the infinitesimal generator of a cosine function $C_0(t)$ with associated sine function $S_0(t)$, and $\wt{B}(t):D(\wt{B}(t))\ra X$ is a closed linear operator with $D\subseteq D(\wt{B}(t))$ for all $t\in I$.

We model this problem in the space $X=L^2(\mb{T},\mb{C})$, where the group $\mb{T}$ is defined as the quotient $\mb{R}/2\pi Z$. We will use the identification between functions on $\mb{T}$ and $2\pi$-periodic functions on $\mb{R}$. Specifically, in what follows we denote by $L^2(\mb{T},\mb{C})$ the space of $2\pi$-periodic 2-integrable functions from $\mb{R}$ into $\mb{C}$. Similarly, $H^2(\mb{T},\mb{C})$ denotes the Sobolev space of $2\pi$-periodic functions $x:\mb{R}\ra \mb{C}$ such that $x''\in L^2(\mb{T},\mb{C})$.

We consider the operator $Ax(\xi)=x''(\xi)$ with domain $D(A)=H^2(\mb{T},\mb{C})$. It is well known that $A$ is the infinitesimal generator of a strongly continuous cosine function $C_0(t)$ on $X$. Moreover, $A$ has discrete spectrum, the spectrum of $A$ consists of eigenvalues $-n^2$ for $n\in \mb{Z}$, with associated eigenvectors $$w_n(\xi)=\frac{1}{\sqrt{2\pi}}e^{in\xi},\quad n\in \mb{Z},$$ the set $\{w_n:n\in Z\}$ is an orthonormal basis of $X$. In particular,
$$Ax=-\sum_{n=1}^\infty n^2\langle x, w_n\rangle w_n $$ for $x\in D(A)$. The cosine function $C_0(t)$ is given by $$C_0(t)x=\sum_{n=1}^\infty\cos(nt)\langle x, w_n\rangle w_n,\quad t\in \mb{R},$$ with associated sine function $$S_0(t)x=\sum_{n=1}^\infty \frac{\sin(nt)}{n}\langle x, w_n\rangle w_n,\quad t\in \mb{R}.$$

It is clear that $\|C_0(t)\|\le 1$ for all $t\in \mb{R}$. Thus, $C(\cdot)$ is uniformly bounded on $\mb{R}$.

Consider the second-order Cauchy problem with control
\begin{align}
\frac{\partial^2}{\partial t^2}z(t,\tau)\in \frac{\partial^2}{\partial \tau^2}z(t,\tau)+b(t)\frac{\partial}{\partial t}z(t,\tau)+\mu(t,\tau)+Q(t,z(t,\tau))\label{vex1}
\end{align}
for $t\in I$, $0\le\tau\le \pi$, subject to the initial conditions
\begin{align}
z(t,0)=&z(t,\pi)=0,\quad t\in I,\label{vex2}\\
z(0,\tau)=&z_0(\tau), \quad 0\le\tau\le \pi,\quad \frac{\partial}{\partial t}z(0,\tau)=z_1(\tau),\label{vex3}
\end{align}
where $b:\mb{R}\ra \mb{R}$, $\mu:I\times [0,\pi]\ra [0,\pi]$ are continuous functions. We fix $a>0$ and set $\beta=\sup_{0\le t\le a}|b(t)|$.

We take $\wt{B}(t)x(\tau)=b(t)x'(\tau)$ defined on $H^1(\mb{T},\mb{C})$. It is easy to see that $A(t)=A+\wt{B}(t)$ is a closed linear operator. Initially we will show that $A+\wt{B}(t)$ generates an evolution operator. It is well known that the solution of the scalar initial value problem
\begin{align*}
q''(t)=&-n^2q(t)+p(t),\\
q(s)=&0,\quad q'(s)=q_1,
\end{align*}
is given by $$q(t)=\frac{q_1}{n}\sin n(t-s)+\frac{1}{n}\int_s^t\sin n(t-\tau)p(\tau)d\tau.$$
Therefore, the solution of the scalar initial value problem
\begin{align}
q''(t)=&-n^2q(t)+in b(t)q(t),\label{vs5e6}\\
q(s)=&0,\quad q'(s)=q_1,\label{vs5e7}
\end{align}
satisfies the integral equation $$q(t)=\frac{q_1}{n}\sin n(t-s)+ i\int_s^t\sin n(t-\tau)b(\tau)q(\tau)d\tau.$$
Applying the Gronwall-Bellman lemma we can affirm that
\begin{align}
|q(t)|\le \frac{|q_1|}{n}e^{\beta(t-s)}\label{vs5e8}
\end{align}
for $s\le t$. We denote by $q_n(t,s)$ the solution of $\eqref{vs5e6}-\eqref{vs5e7}$. We define $$S(t,s)x=\sum_{n=1}^\infty q_n(t,s)\langle x, w_n\rangle w_n.$$

It follows from the estimate $\eqref{vs5e8}$ that $S(t,s):X\ra X$ is well defined and satisfies the conditions of Definition \ref{vd1}.

Let us define$f:I\times X\ra X$ be defined as $$F(t,z)(\tau)=Q(t,z(\tau)),\quad z\in X,\quad \tau\in [0,\pi].$$ and $u:I\ra U$ be defined as $$B(u(t))(\tau)=\mu(t,\tau),\quad \tau\in [0,\pi],$$ where $\mu:I\times [0,\pi]\ra [0,\pi]$ is continuous.

Assume these functions satisfy the requirement of hypotheses. From the above choices of the functions and evolution operator $A(t)$ with $B=I$ , the system \eqref{vex1}-\eqref{vex3} can be formulated as an abstract second order semilinear system \eqref{vi1}-\eqref{vi2} in $X$. Since all hypotheses of Theorem \ref{40} are satisfied, approximate controllability of system \eqref{vex1}-\eqref{vex3} on $I$ follows from Theorem \ref{40}.


\begin{thebibliography}{USA00}

\bibitem{na1} N. Abada, M. Benchohra and H. Hammouche, Existence and controllability results for nondensely defined impulsive semilinear functional differential inclusions, {\it J. Diff. Equ.} {\bf 246(10)} (2009), 3834-3863.

\bibitem{sa1} S. Aizicovici and M. McKibben, Existence results for a class of abstract nonlocal Cauchy problems, {\it Nonlinear Anal.} {\bf 39} (2000), 649-668.

\bibitem{sa2} S. Aizicovici and H. Lee, Nonlinear nonlocal Cauchy problems in Banach spaces, {\it Appl. Math. Lett.} {\bf 18} (2005), 401-407.

\bibitem{ga1} G. Arthi and K. Balachandran, Controllability of second-order impulsive evolution systems with infinite delay, {Nonlinera Anal. Hybrid Sys.} {\bf 11} (2014), 139-153.

\bibitem{ga2} G. Arthi, J.H. Park and H.Y. Jung, Existence and controllability results for second-order impulsive stochastic evolution systems with state-dependent delay, {\it Appl. Math. Comput.} {\bf 248} (2014), 328-341.

\bibitem{ddb1} D.D. Bainov and P.S. Simeonov, {\it Impulsive Differential Equations: Periodic Solutions and Applications}, Longman Scientific and Technical Group, England, 1993.

\bibitem{cjk1} C.J.K. Batty, R. Chill and S. Srivastava, Maximal regularity for second order non-autonomous Cauchy problems, {\it Studia Math.} {\bf 189} (2008) 205-223.

\bibitem{mb1} M. Benchohra, J. Henderson and S.K. Ntouyas,{\it Impulsive Differential Equations and Inclusions, Contemporary Mathematics and Its Applications}, Hindawi Publishing Corporation, New York, 2006.

\bibitem{hfb1} H.F. Bohnenblust and S. Karlin, {\it On a Theorem of Ville, in: Contributions to the Theory of Games}, vol. I, Princeton University Press, Princeton, NJ, 1950, pp. 155-160.

\bibitem{lb1} L. Byszewski, Theorems about the existence and uniqueness of solutions of a semilinear evolution nonlocal Cauchy problem, {\it J. Math. Anal. Appl.} {\bf 162} (1991), 494-505.

\bibitem{lb2} L. Byszewski and H. Akca, On a mild solution of a semilinear functional-differential evolution nonlocal problem, {\it J. Appl. Math. Stoc. Anal.} {\bf 10(3)} (1997), 265-271.

\bibitem{lb3} L. Byszewski, V. Lakshmikantham, Theorem about the existence and uniqueness of solutions of a nonlocal Cauchy problem in a Banach space, {\it Appl.
Anal.} 40 (1990), 11-19.

\bibitem{jpd1} J.P. Dauer, N.I. Mahmudov and M.M. Matar, Approximate controllability of backward stochastic evolution equations in Hilbert spaces, {\it J.  Math. Anal. Appl.} {\bf 323(1)} (2006), 42-56.

\bibitem{kde1} K. Deimling, {\it Multivalued Differential Equations}, De Gruyter, Berlin, 1992.

\bibitem{ke1} K. Ezzinbi, X. Fu and K. Hilal, Existence and regularity in the a-norm for some neutral partial differential equations with nonlocal conditions, {\it Nonlinear Anal.} {\bf 67} (2007), 1613-1622.

\bibitem{ff1} F. Faraci and A. Iannizzotto, A multiplicity theorem for a perturbed second-order non-autonomous system, {\it Proc. Edinb. Math. Soc.} {\bf 49} (2006), 267-275.

\bibitem{hof1} H. O. Fattorini; {\it Second Order Linear Differential Equations in Banach Spaces}, North-Holland Math. Stud., 108, Amsterdam: North-Holland, 1985.

\bibitem{xf1} X. Fu and K. Ezzinbi, Existence of solutions for neutral functional differential evolution equations with nonlocal conditions, {\it Nonlinear Anal.} {\bf 54} (2003), 215-227.

\bibitem{tg1} T. Guendouzi and L. Bousmaha, Approximate controllability of fractional neutral stochastic functional integro-differential inclusions with infinite delay, {\it Qual. Theory Dyn. Syst.} {\bf 13} (2014), 89-119.

\bibitem{eh1} E. Hern\'{a}ndez, M. McKibben, Some comments on: Existence of solutions of abstract nonlinear second-order neutral functional integrodifferential equations, {\it Comput. Math. Appl.} {\bf 50 } (2005), 655-669.

\bibitem{hrh1} H.R. Henr\'{i}quez , Existence of solutions of non-autonomous second order functional differential equations with infinite delay, {\it Nonlinear Anal.: TMA}, {bf 74} (2011), 3333-3352.

\bibitem{hrh2} H.R. Henr\'{i}quez and Genaro Castillo, The Kneser property for the second order functional abstract Cauchy problem, {\it Integ. Equ. Oper. Theo.}  {\bf 52} (2005), 505-525.

\bibitem{hrh3} H.R. Henr\'{i}quez and E. Hern\'{a}ndez, Existence of solutions of a second order abstract functional Cauchy problem with nonlocal conditions,\ {\it Annal. Polon. Math.} {\bf 88 (2)} (2006), 141-159.

\bibitem{hrh4} H.R. Henr\'{i}quez, Approximate controllability of linear distributed control systems, {\it Appl. Math. Lett.} {bf 21(10)} (2008), 1041-1045.

\bibitem{jk1} J. Kisy\'{n}ski, {\it On cosine operator functions and one parameter group of operators},\ {\it Studia. Math.} {\bf 49} (1972), 93-105.

\bibitem{mk1} M. Kozak, A fundamental solution of a second order differential equation in a Banach space, {\it Univ. Iagellon. Acta Math.} {\bf 32} (1995), 275-289.

\bibitem{vl1} V. Laksmikantham, D. Bainov and P. S. Simenov, {\it Theory of impulsive differential equations}, Series in Modern Applied Mathematics, 6. World Scientific Publishing Co., Inc., Teaneck, NJ, 1989.

\bibitem{al1} A. Lasota and Z. Opial, An application of the Kakutani-Ky-Fan theorem in the theory of ordinary differential equations or noncompact acyclic-valued map, {\it Bull. Acad. Pol. Sci. Ser. Sci. Math. Astronom. Phys.} {\bf 13} (1965), 781-786.

\bibitem{jl1} J. Liang, J.H. Liu and T.J. Xiao, Nonlocal Cauchy problems governed by compact operator families, {\it Nonlinear Anal.} {\bf 57} (2004), 183-189.

\bibitem{yl1} Y. Lin, Time-dependent perturbation theory for abstract evolution equations of second order, {\it Stud. Math.} {\bf 130} (1998), 263-274.

\bibitem{bmm1} B.M. Miller and E.Y. Rubinovich, {\it Impulsive Control in Continuous and Discrete-Continuous Systems}, Kluwer Academic/Plenum Publishers, New York, 2003.

\bibitem{nim1} N.I. Mahmudov and A. Denker, On controllability of linear stochastic systems, {\it Int. J. Cont.} {\bf 73} (2000), 144-151.

\bibitem{nim2} N.I. Mahmudov, Approximate controllability of evolution systems with nonlocal conditions, {\it Nonlinear Anal.: TMA}, {\bf 68(3)} (2008), 536-546.

\bibitem{nim4} N.I. Mahmudov and S Zorlu, On the approximate controllability of fractional evolution equations with compact analytic semigroup, {\it J. Comput. Appl. Math.} {\bf 259} (2014), 194-204.

\bibitem{nim5} N.I. Mahmudov, Approximate controllability of some nonlinear systems in Banach spaces, {\it Bound. val. Prob} {\bf 2013(1)} (2013), 1-13.

\bibitem{nim7} N.I. Mahmudov and A. Denker, On controllability of linear stochastic systems, {\it Int. J. Control} {\bf 73} (2000), 144-151.

\bibitem{nim8} N.I. Mahmudov and M.A. McKibben, Approximate controllability of second-order neutral stochastic evolution equations, {\it Dyn. Contin. Discrete Impuls. Syst.} {\bf 13(5)} (2006),  619-634.

\bibitem{eo1} E. Obrecht, Evolution operators for higher order abstract parabolic equations, {\it Czech. Math. Jour.} {\bf 36} (1986), 210-222.

\bibitem{yp1} Y. Peng and X. Xiang, Second-order nonlinear impulsive time-variant systems with unbounded perturbation and optimal controls, {\it J. Ind. Manag. Optim.} {\bf4} (2008), 17-32.

\bibitem{ams1}  A. M. Samoilenko and N.A. Perestyuk; {\it Impulsive  Differential Equations}, World Scientific, Singapore, 1995.

\bibitem{rs6} R. Sakthivel, R. Ganesh and S.M. Anthoni, Approximate controllability of fractional nonlinear differential inclusions, {\it Appl. math. comput.} {\bf 225} (2013), 708-717.

\bibitem{rs1} R. Sakthivel, N.I. Mahmudov and J.H. Kim, Approximate controllability of nonlinear impulsive differential systems, {\it Report. Math. Phy.} {\bf 60(1)} (2007), 85-96.

\bibitem{rs2} R. Sakthivel, E.R. Anandhi and N. I. Mahmudov, Approximate controllability of second-order systems with state-dependent delay, {\it Numer. Funct. Anal. Optim.} {\bf 29(11-12)} (2008), 1347-1362.

\bibitem{hs1} H. Serizawa, M. Watanabe, Time-dependent perturbation for cosine families in Banach spaces, {\it Hous. J. Math.} {\bf 12} (1986), 579-586.

\bibitem{sh1} S. Hu and N.S. Papageorgiou, {\it Handbook of Multivalued Analysis (Theory)}, Kluwer Academic Publishers, Dordrecht Boston, London, 1997.

\bibitem{ss1} S. Sivasankaran, M. Mallika Arjunan and V. Vijayakumar, Existence of global solutions for second order impulsive abstract partial differential equations, {\it Nonlinear Anal. TMA}, {\bf 74(17)} (2011), 6747-6757.

\bibitem{cct1} C. C. Travis and G. F. Webb, Compactness, regularity, and uniform continuity properties of strongly continuous cosine families, \ {\it Houston J. Math.} {\bf 3(4)} (1977), 555-567.

\bibitem{cct2} C. C. Travis and G. F. Webb, Cosine families and abstract nonlinear second order differential equations, {\it Acta. Math. Acad. Sci. Hungar.} {\bf 32} (1978), 76-96.

\bibitem{vv3} V. Vijayakumar, S. Sivasankaran and M. Mallika Arjunan, Existence of solutions for second order impulsive partial neutral functional integrodifferential equations with infinite delay, {\it Nonlinear Stud.} {\bf 19(2)} (2012), 327-343.

\bibitem{vv1} V. Vijayakumar, S. Sivasankaran and M. Mallika Arjunan, Existence of global solutions for second order impulsive abstract functional integrodifferential equations, {\it Dyn. Contin. Discrete Impuls. Syst.} {\bf 18} (2011), 747-766.

\bibitem{vv5} V. Vijayakumar, C. Ravichandran and R. Murugesu, Nonlocal controllability of mixed Volterra-Fredholm type fractional semilinear integro-differential inclusions in Banach spaces, {\it Dyn. Contin. Discrete Impuls. Syst.} {\bf 20(4)} (2013), 485-502.

\bibitem{vv6} V. Vijayakumar, C. Ravichandran and R. Murugesu, Approximate controllability for a class of fractional neutral integro-differential inclusions with state-dependent delay, {\it Nonlinear stud.} {\bf 20(4)} (2013), 511-530.

\bibitem{xx1} X. Xue, Existence of solutions for semilinear nonlocal Cauchy problems in Banach spaces, {\it Elect. J. Diff. Equ.} {\bf 64} (2005), 1-7.

\bibitem{zy1} Z. Yan, Approximate controllability of fractional neutral integro-differential inclusions with state-dependent delay in Hilbert spaces, {\it IMA J. Math. Cont. Inform.} (2012), doi:10.1093/imamci/dns033.

\bibitem{zy2} Z. Yan, Approximate controllability of partial neutral functional differential systems of fractional order with state-dependent delay, {\it Intern. J. Cont.} {\bf 85(8)} (2012), 1051-1062.


\end{thebibliography}
\end{document}